\documentclass[12pt,onecolumn]{IEEEtran}
\usepackage{amsfonts}
\usepackage{diagbox}
\usepackage{hyperref}
\usepackage{fancyhdr}
\usepackage[all]{xy}
\usepackage{mathrsfs}
\usepackage[latin1]{inputenc}
\usepackage[usenames,dvipsnames]{color}
\usepackage{makeidx}
\usepackage{graphicx}
\usepackage{amsthm}
\usepackage{amssymb}
\usepackage{amsmath}
 \usepackage{multirow}

\makeindex

\newcommand{\notransversal}{\mathbin{\hbox{$\cap$ \kern-.45cm
\raise.3ex\hbox{$\top$}\kern-.26cm \raise.3ex\hbox{$/$}}}}

\newtheorem{theorem}{Theorem}
\newtheorem{corollary}[theorem]{Corollary}
\newtheorem{lemma}[theorem]{Lemma}
\newtheorem{proposition}[theorem]{Proposition}
\newtheorem{example}[theorem]{Example}
\newtheorem{definition}[theorem]{Definition}
\newtheorem{remark}[theorem]{Remark}

\vspace{25cm}
\begin{document}
\thispagestyle{empty}

\pagenumbering{roman} 

\pagestyle{headings}
\title{Linear  Permutations and their Compositional Inverses over $\mathbb{F}_{q^n}$}
\author{Gustavo Terra Bastos}

\maketitle

\begin{abstract}
The use of permutation polynomials has appeared, along to their compositional inverses, as a good choice in the implementation of cryptographic systems. Hence, there has been a demand for constructions of these polynomials which coefficients belong to a finite field. As a particular case of permutation polynomial, involution is highly desired since its compositional inverse is itself. In this work, we present an effective way of how to construct several linear permutation polynomials over $\mathbb{F}_{q^n}$ as well as their compositional inverses using a decomposition of $\displaystyle{\frac{\mathbb{F}_q[x]}{\left\langle x^n -1 \right\rangle}}$ based on its primitive idempotents. As a consequence, an immediate construction of involutions is presented.
\end{abstract}

\begin{IEEEkeywords}
Finite fields, linear permutations, permutation polynomials, involutions.
\end{IEEEkeywords}

%
\IEEEpeerreviewmaketitle

\section{Introduction}
%
%
%
%
\IEEEPARstart{L}{et} $\mathbb{F}_{q}$ be a finite field with $q$ elements, in which $q$ is prime or a prime power. $\mathbb{F}_{q^n}$ denotes a finite extension field of $\mathbb{F}_{q}$, which may be seen as a $n-$dimensional vector space over $\mathbb{F}_q$.

A polynomial $f(x) \in \mathbb{F}_{q^n} [x]$ is called a \emph{permutation polynomial} of $\mathbb{F}_{q^n}$ if the induced mapping by it
\begin{equation}
\begin{array}{cccc}
\phi_f:&\mathbb{F}_{q^n} &\rightarrow& \mathbb{F}_{q^n}\\
&x &\mapsto& \phi_f (x)=f(x)
\end{array}  
\end{equation}
is a bijection of $\mathbb{F}_{q^n}$ on itself. Note that it is always possible to obtain $f(x) \in \mathbb{F}_{q^n} [x]$ from $\phi_f$ by the Lagrange interpolation method.

From the finiteness of $\mathbb{F}_{q^n}$, simple conditions determine if $f(x)$ is a permutation polynomial; for instance, if $f(x)$ is one-to-one. Nonetheless, making explicit conditions over the coefficients of $f(x)$ so that it is a permutation one is not an easy task.

Given $f(x)$ a bijection/permutation polynomial of $\mathbb{F}_{q^n}$, the (unique) compositional inverse of $f(x)$ is denoted $f^{-1} (x) \in \mathbb{F}_{q^n} [x]$, in which 
\begin{equation}
f(x) \circ f^{-1}(x)\equiv f^{-1} (x) \circ f(x) \equiv x \mod x^{q^n} -x .  
\end{equation}

Recently, in~\cite{china} the authors have presented the compositional inverses of all permutation polynomials of degree $\leq 6$ over $\mathbb{F}_{q^n}$ and inverses of permutation polynomial of degree $7$ in characteristic $2$.

Based on the concept of complete mappings for groups, complete permutation polynomials~\cite{completemaps} are defined. If $f(x) \in \mathbb{F}_{q^n} [x]$ is a permutation polynomial over $\mathbb{F}_{q^n}$ in which $f(x)+x$ is also a permutation one, so $f(x)$ is called \emph{complete permutation polynomial} (or 
complete mapping polynomial).  In~\cite{completemaps} several families of complete permutation polynomials have been presented and, in particular, all complete permutation polynomials which degree is less than 6 have been classified. Later~\cite{bcpp}, $\lambda-$\emph{complete permutation polynomials}, which are the natural extension of complete permutation polynomial have been defined, namely, given $\lambda \in \mathbb{F}_{q^n}$, $f(x)$ is a $\lambda-$complete permutation polynomial if $f(x)$ and $f(x)+\lambda x$ are permutation polynomials over $\mathbb{F}_{q^n}$.

In~\cite{compinverse2} it has been presented the compositional inverses of linear permutations (particular case of permutation polynomials) of the form $x + x^2 + Tr_{2^n}\left(\frac{x}{a}\right)$, in which $a\in \mathbb{F}^{*}_{2^n}$ and $Tr_{2^n}(\cdot)$ denotes the Trace map over $\mathbb{F}_{2^n}$. In~\cite{tuxanidy}, the authors have exhibited compositional inverses of some linear permutation binomials beyond some $\lambda-$complete permutation polynomials.

In particular, permutation polynomials, which compositional inverses are themselves, are called \emph{involutions}, i.e, $f(x) \in \mathbb{F}_{q^n}[x]$ is an involution if $(f\circ f)(x)=f^2 (x)\equiv x \mod x^{q^n} -x$. 

Describing explicit families of permutation polynomials and their compositional inverses is a current research problem both to theoretic aspect and from application perspective due to so many interesting issues in  error-correcting codes, cryptography and combinatorial designs. Just quoting a reference about the relevance of permutation polynomial studies, it is worth mentioning~\cite{app} in which, in 1970s, the authors already used permutation polynomials/rational functions over finite fields in order to propose some cryptographic systems. Currently, permutation polynomials may be applied to S-boxes in cryptosystems acting as extra protection layer, and their compositional inverses working on decryption process. The use of involutions in this context appears as interesting solution once the system is not required to storage the different permutations to the encryption-decryption process. For more recent references, see PRINCE~\cite{prince} (Use of linear involutions) and iSCREAM~\cite{grosso} (Use of non-linear involutions).

Based on applications of involutions in cryptography problems and even the development of the own theory properly, in~\cite{charpin} the authors have developed a mathematical background for involutions providing  several constructions over $2-$characteristic finite fields, making use of linear polynomials and $b-$linear translators as some of the algebraic tools used in that paper. Moreover, an analysis about fixed points for some involutions is also addressed.

In~\cite{lucas} the author has proposed some families of linear permutations over $\mathbb{F}_{q^n}$ by making use of a new class of linear polynomials called \emph{nilpotent linear polynomials} which are defined next: Given a positive integer $t\geq 2$, $L(x)\in \mathbb{F}_{q^n} [x]$ is a nilpotent linear polynomial if $L^t (x) =(\underbrace{L\circ L \circ...\circ L}_{t-times})(x)\equiv 0 \mod x^{q^n} - x$. He has also proposed constructions of binary linear involutions with no fixed points.

A characterization of when $x^r h\left(x^s \right)$ is an involution over $\mathbb{F}_q$ has been developed in~\cite{zheng} under some restrictions over $r,s$ and $h(x)$. It is obtained from involutions over the set of $d-$th roots of unity, in which $ds=q-1$, and congruent and linear equation systems. 

From the AGW-Criterion~\cite{agw}, in~\cite{niu} some involutions with form $x^r h(x^{q-1})$ over $\mathbb{F}_{q^2}$ have been constructed. Moreover, the authors have approached how to explicit the compositional inverse and, in particular cases, the involutory property of polynomial $f(x)=g\left(x^{q^i}- x +\delta\right) +cx$. Finally, the fixed points of some polynomials described in this paper have also been analyzed.

This work aims to study the behaviour of conventional $q-$associates of linear permutations on the simple components of the $\mathbb{F}_q-$algebra $\displaystyle{R_{q,n}:=\frac{\mathbb{F}_q [x]}{\left\langle x^n -1 \right\rangle}}$ from the primitive idempotent perspective. In particular cases, primitive idempotents are easily described via $\mathbb{F}_q-$algebra isomorphism between $R_{q,n}$ and the group algebra $\mathbb{F}_q C$, in which $C$ is $n-$order cyclic group. Based on in this new perspective and from the possession of some units in $R_{q,n}$, it is offered an easy implementation technique which allow us to describe families of linear permutations and their respective (linear) compositional inverses. In particular, families of involutions are also described, which elements can be used once more in order to provide new involutions.  

In Section~\ref{basics}, we review the mathematical background needed for understanding linear polynomials (and the particular cases: linear permutations and involutions) and the relationship between linear polynomials and their conventional $q-$associates. Such relationship is explored throughout this work. Further, we define primitive idempotents and  analyze some properties. In Section~\ref{sectioni}, we study the cyclic shifts of linear permutations, that are linear ones as well, and which applications will be used in Section~\ref{construcao}. In this Section, we present linear permutations over $\mathbb{F}_{q^n}$ which coefficients are in $\mathbb{F}_{q}$ from a rereading of the famous Chinese Remainder Theorem based on primitive idempotents. As a particular case, it is possible to get several linear involutions over $\mathbb{F}_{q^n}$. Finally, conclusions and future problems are drawn in Section~\ref{conclusion}.
\section{Basics}\label{basics}
\subsection{Linear Polynomials}
Given $f(x) \in \mathbb{F}_{q^n} [x]$ a permutation polynomial over $\mathbb{F}_{q^n}$, there exists $g(x)\in \mathbb{F}_{q^n} [x]$ which degree is less than $q^n$ and $f(c)=g(c)$ for all $c \in \mathbb{F}_{q^n}$, namely, $f(x)\equiv g(x)\mod x^{q^n} -x$. Thus, it is possible to represent every permutation polynomials in a \emph{reduced degree} version, i.e., permutation polynomials which degree is less than $q^n$. From now on, we will implicitly work only with reduced degree polynomials.

In this work, we focus on the class of the \emph{linear/linearized polynomials} (also called $q$-polynomials) over $\mathbb{F}_{q^n}$, which also represent the $\mathbb{F}_q-$linear mappings from $\mathbb{F}_{q^n}$ to itself, seen as $n-$dimensional $\mathbb{F}_q-$vector space. Such polynomials are described as
\begin{equation}
F(x)=\sum_{i=0}^{n-1}f_i x^{[i]}\in \mathbb{F}_{q^n} [x],    
\end{equation}
in which $[i]=q^i$, for $0\leq i \leq n-1$. We refer to permutation linear polynomials simply as \emph{linear permutations}.

The following properties are verified for linear polynomials over $\mathbb{F}_{q^n}$  
\begin{itemize}
\item[(i)] $F\left(\alpha + \beta\right)=F\left(\alpha \right)+ F\left(\beta\right)$ and 
\item[(ii)] $F\left(a\alpha\right)= a F\left(\alpha\right)$, 
\end{itemize}
for all $\alpha,\beta \in \mathbb{F}_{q^m}$ and $a \in \mathbb{F}_q$, in which $\mathbb{F}_{q^{m}}$ is an arbitrary extension of $\mathbb{F}_{q^n}$. For a seminal reference over finite fields and, in particular, polynomials over finite fields see~\cite{finitefields}.

Besides the applications of (permutation) linear polynomials in block codes, cryptography and combinatorial designs, currently one of their subclass called \emph{subspace polynomials} has been applied in the context of random network coding~\cite{koetterk} in order to obtain good constant dimension codes, which are $k-$dimensional subspace codes. See~\cite{zhao} for one of the most recent constructions and the references therein.

Let 
\begin{eqnarray}
\mathcal{L}_n \left(\mathbb{F}_{q^n}\right)&:=&\frac{\mathbb{F}_{q^n}[x]}{\left\langle x^{q^n} - x \right\rangle}\\
&:=&\left\{f_{n-1}x^{[n-1]}+...+f_0 x^{[0]}: f_i \in \mathbb{F}_{q^n}\right\} \nonumber
\end{eqnarray} 
be the set of linear polynomials over $\mathbb{F}_{q^n}$. In particular, the subset of $\mathcal{L}_n \left(\mathbb{F}_{q^n}\right)$ formed by polynomials which coefficients belong to $\mathbb{F}_{q}$ is denoted as $\mathcal{L}_n \left(\mathbb{F}_{q}\right)$. Since the ordinary multiplication of two linear polynomials, in general, do not provide a linear polynomial, then one define the \emph{symbolic multiplication} between two linear polynomials $F(x)$ and $G(x)$ as $F(x) \circ G (x)=F\left(G (x)\right)$, namely, the symbolic multiplication is, in fact, the composition operation. In this sense, considering $F(x), G (x)\in \mathcal{L}_n \left(\mathbb{F}_{q^n}\right)$, $G (x)$ \emph{divides symbolically} $F(x)$ if there is a linear polynomial $H(x)\in \mathcal{L}_n \left(\mathbb{F}_{q^n}\right)$ so that $G(x)\circ H(x)=F(x)$. It is trivial to notice that $\circ$ is not a commutative operation, but it is associative. Hence, in possession of the usual sum and scalar product, besides the polynomial composition/symbolic multiplication, the set $\mathcal{L}_n \left(\mathbb{F}_{q^n}\right)$ is in fact a non-commutative $\mathbb{F}_q$-algebra. In particular, $\mathcal{P}_n \left(\mathbb{F}_{q^n}\right)\subset\mathcal{L}_n \left(\mathbb{F}_{q^n}\right)$ describes the non-abelian group under the composition operation formed by the linear permutations. For more information about the algebraic structure of $\mathcal{L}_n \left(\mathbb{F}_{q^n}\right)$ and its several isomorphic forms as group algebra, matrix algebra and etc., we recommend~\cite{wuliu}.

\begin{definition}
The polynomials 
\begin{equation}
f(x)=\sum_{i=0} ^{n-1} f_i x^i\,\,\,\mbox{and}\,\,\, F(x)=\sum_{i=0} ^{n-1} f_i x^{[i]}
\end{equation}
over $\mathbb{F}_{q^n}$ are called $q-$associates of each other. More specifically, $f(x)$ is \emph{the conventional $q-$associate} of $F(x)$ and $F(x)$ is \emph{the linearized $q-$associate} of $f(x)$.
\end{definition}

Henceforth, we denote linear polynomials with capital letters and their conventional \linebreak$q-$associates with small letters.

\begin{lemma}\cite[Lemma 3.59]{finitefields}\label{lemma2}
Let $F (x)$ and $G (x)$ be linear polynomials over $\mathbb{F}_q$ with conventional $q-$associates $f (x)$ and $g (x)$. Then $h(x)=f (x) g (x)$ and $H(x)=F(x) \circ G (x)$ are $q-$associates of each other.
\end{lemma}

\begin{theorem}\cite[Theorem 3.62]{finitefields}\label{divisao}
Let $F (x)$ and $G(x)$ be linear polynomials over $\mathbb{F}_q$ with conventional $q-$associates $f (x)$ and $g(x)$. Then the following properties are equivalent: 
\begin{itemize}
\item[(i)] $F(x)$ symbolically divides $G(x)$;
\item[(ii)] $F (x)$ divides $G(x)$ in the ordinary sense; 
\item[(iii)] $f (x)$ divides $g(x)$.
\end{itemize}
\end{theorem}

It is important to stress that the Lemma~\ref{lemma2} and Theorem~\ref{divisao} will be used freely during all this work.

Another useful theoretical consideration used in this work comes from the classical Rank-Nullity Theorem, in which the linear polynomial $F(x) \in \mathcal{L}_n \left(\mathbb{F}_{q^n}\right)$ is a linear permutation if and only if $0$ is its only root in $\mathbb{F}_{q^n}$.

\subsection{Primitive Idempotents}
In~\cite{polcycliccodes} the authors used linear polynomials in order to get good cyclic codes~\cite{mac}. In this work, we somehow take the opposite way, i.e, the idempotent generators of cyclic codes are used in order to provide a linear permutation (linear polynomial) construction.

From now on, we take positive integers $n$ and $q$ so that $\gcd(n,q)=1$. Based on well-known Chinese Remainder Theorem, the ring $\displaystyle{R_{q,n} =\frac{\mathbb{F}_q [x]}{\left\langle x^n -1 \right\rangle}}$ can be decomposed as
\begin{eqnarray}\label{tcr}
R_{q,n} =\frac{\mathbb{F}_q [x]}{\left\langle x^n -1 \right\rangle}&\cong&  \frac{\mathbb{F}_q [x]}{\left\langle f_1 (x) \right\rangle} \oplus \frac{\mathbb{F}_q [x]}{\left\langle f_2 (x) \right\rangle} \oplus \ldots \oplus \frac{\mathbb{F}_q [x]}{\left\langle f_t (x) \right\rangle},\nonumber\\
&\cong& \mathbb{F}_q\left(\xi_1 \right)\oplus\mathbb{F}_q\left(\xi_2 \right)\oplus\ldots \mathbb{F}_q\left(\xi_t \right)
\end{eqnarray}
in which $f_i (x)$ are the distinct irreducible factors of $x^n -1$ and $\mathbb{F}_q \left(\xi_i \right)$ are finite extensions of $\mathbb{F}_q$ given by roots of $f_i (x)$, for $1\leq i \leq t$. The minimal ideals $\displaystyle{\frac{\mathbb{F}_q [x]}{\left\langle f_i (x)\right\rangle}}$ $\left(\mbox{or } \mathbb{F}_q \left(\xi_i\right)\right)$ in the decomposition of $R_{q,n}$ in~\eqref{tcr} are called its \emph{simple components}. 
\begin{definition}
A polynomial $e(x)$ of $R_{q,n}$ is an \emph{idempotent} if $e(x)\equiv e^2 (x)=e(x) e(x) \mod x^n -1.$
\end{definition}

Since $R_{q,n}$ is a semisimple ring~\cite{polcino}, there exists a family $\{e_1 (x) , e_2 (x) , ... ,e_t (x)\}$ of non-zero elements in $R_{q,n}$, called (orthogonal) \emph{primitive idempotents of} $R_{q,n}$, so that
\begin{itemize}
\item [(i)] If $i \neq j $, then 
\begin{equation}\label{eq7}
e_i (x)e_j (x) \equiv 0\mod x^n -1,
\end{equation}
for $1 \leq i\neq  j\leq t$. \item [(ii)] 
\begin{equation}\label{soma1}
e_1 (x) + e_2 (x) + ...+ e_t (x)=1.
\end{equation}
\item [(iii)] $e_i(x)$ cannot be written as $e_i (x) = \overline{e} (x) + \tilde{e} (x),$ in which $\overline{e} (x)$ and $\tilde{e} (x)$ are non-zero idempotents so that $\overline{e} (x) \tilde{e} (x)=0$, $1 \leq i\leq t$. 
\end{itemize}
%

\begin{remark}
Notice that all non-primitive idempotents $e(x)$ can be written as a sum of some primitive idempotents.  
\end{remark}

Idempotents and divisors of $x^n -1$ in $R_{q,n}$ are related. On next result, given $g(x)$ a divisor of $x^n -1$ in $R_{q,n}$, we have

\begin{theorem}\cite[Theorem 1,\, pag. 217]{mac}\label{idempcodc}
\begin{itemize}
\item[(i)] A cyclic code or ideal $\mathcal{C}=\langle g(x)\rangle$ contains a unique idempotent $e(x)$ so that $\mathcal{C}=\langle e(x)\rangle$. Also $e(x)=p(x)g(x)$ for some polynomial $p(x)$, and $e\left(\alpha^i \right)=0$ iff $g\left(\alpha^i \right)=0$.
\item[(ii)] $c(x) \in \mathcal{C}$ if and only if $c(x)e(x)=c(x)$.
\end{itemize}
\end{theorem}

More information and results about semisimple rings and other considerations about their algebraic structure based on idempotents, see~\cite{polcino}.

\begin{remark}
Notice that the simple components of $R_{q,n}$ are generated by primitive idempotents as well. For more information, see~\cite[pg.219]{mac}.
\end{remark}
\section{$\alpha-$cyclic shift and linear permutations}\label{sectioni}
The motivation of this section comes from the fact that composition of linear permutations is also a linear permutation. From an equivalence relation defined next, we provide large trivially-constructed families of  linear permutations. We also analyze under which conditions such equivalence relation can be used, in order to get the corresponding compositional inverses, in particular, involutions over $\mathbb{F}_{q^n}$. The results of this section may be applied to further constructions as it will be seen on next Section.

The following basic result will guide all the discussions proposed in this section 

\begin{proposition}\label{deslciclico}
If $F(x) \in \mathcal{L}_n\left(\mathbb{F}_{q^n}\right)$ is a linear permutation over $\mathbb{F}_{q^n}$, then $F(x)\circ \alpha x^{[1]}\in \mathcal{L}_n\left(\mathbb{F}_{q^n}\right)$ is also a linear permutation of $\mathbb{F}_{q^n}$, for any $\alpha \in \mathbb{F}^{*}_{q^n}$.
\end{proposition}

\begin{proof}
Since $gcd\left(q^n -1, q \right)=1$, it follows from~\cite[Theorem 7.8, (ii)]{finitefields} that the monomial $x^{[1]}$ is a linear permutation of $\mathbb{F}_{q^n}$. Hence, $\alpha x^{[1]}$ so is for any $\alpha \in \mathbb{F}^{*}_{q^n}$. Therefore, since $F(x)\circ \alpha x^{[1]}$ is a composition of linear permutations, the result follows.
\end{proof}

Hence, if $\displaystyle{F(x)=\sum_{i=0}^{n-1}f_i x^{[i]}}\in \mathcal{L}_{n}\left(\mathbb{F}_{q^n}\right)$ and $\alpha \in \mathbb{F}^{*}_{q^n}$, then 
\begin{equation}
F(x)\circ \alpha x^{[1]}=\sum_{i=0}^{n-1}\alpha^{[i]} f_i x^{[i+1]},    
\end{equation}
in which the superscripts are taken modulo $n$.

In~\cite{gilbert}, it has been introduced the concept of \emph{cyclically permutable codes}, which are $n-$length block codes in $\mathbb{F}_2 ^n$ . For a more general definition of cyclically permutable codes and their applications in some communication problems, see~\cite{valdemar} and references therein.

In this section, we adapt the classical definitions of cyclic order and cyclic equivalence class arising from cyclically permutable codes constructed in $\displaystyle{R_{q,n} }$ to linear polynomials in $\displaystyle{\mathcal{L}_n \left(\mathbb{F}_{q^n}\right)}$. Originally, let $c(x) \in R_{q,n}$. If $S(\cdot)$ denotes the cyclic shift operator, i.e., $S(c(x))= xc(x)$ taken modulo $x^n -1$, then there is a least integer $1\leq m \leq n$ so that $S^m(c(x)) = S^{m-1}\left(S(c(x))\right)\equiv c(x) \mod x^n -1$, which it is called \emph{cyclic order} of $c(x)$. If $d(x)\equiv x^t c(x) \mod x^n -1$, $1\leq t\leq n-1$, then $d(x)$ is \emph{cyclically equivalent} to $c(x)$ in $R_{q,n}$. Extending the definitions above to $\displaystyle{\mathcal{L}_n \left(\mathbb{F}_{q^n}\right)}-$context, denote $S_{\alpha}(\cdot)$ the (left) $\alpha-$cyclic shift operator, for some $\alpha \in \mathbb{F}^{*}_{q^n}$, as $S_{\alpha}(F(x))\equiv F(x)\circ \alpha x^{[1]}\mod x^{[n]} -x$. In particular, for $\alpha=1$ the corresponding cyclic shift in $\displaystyle{\mathcal{L}_n \left(\mathbb{F}_{q^n}\right)}$ is equal to its equivalent in $R_{q,n}$ and we will adopt the same notation. Given the least integer $m$ so that $S_{\alpha}^m(F(x))=\underbrace{\left(S_{\alpha} \circ ...\circ S_{\alpha}\right)}_{m-times}(F(x)) \equiv F(x) \mod x^{[n]}-x$ such integer $m$ is called as $\alpha-$\emph{cyclic order of} $F(x)$. For $F(x)$ a linear permutation over $\mathbb{F}_{q^n}$, this $m$ is computed below

\begin{theorem}\label{ordciclica}
Let $\beta$ be a primitive element of $\mathbb{F}_{q^n}$. If $F(x) \in \mathcal{L}_n\left(\mathbb{F}_{q^n}\right)$ is a linear permutation over $\mathbb{F}_{q^n}$, $\alpha = \beta^l$, for $1\leq l \leq q^n -2$, and $t$ the least positive integer so that $lt\equiv 0\mod q-1$, then the $\alpha-$cyclic order of $F(x)$ is $m=tn$.
\end{theorem}
\begin{proof}
Let $t$ be the least positive integer so that $lt\equiv 0\mod q-1$. Since $F(x)$ is a linear permutation, it is clear that the $1-$cyclic order of $F(x)$ is $n$ because, given least positive integer $m\leq n$ so that $F(x)\circ x^{[m]}\equiv F(x)\mod x^{[n]}-x$, we have
\begin{eqnarray}
x^{[m]}&\equiv& F^{-1}(x)\circ F(x)\circ x^{[m]} \equiv F^{-1}(x)\circ F(x) \nonumber\\
&\equiv& x\mod x^{[n]}-x, 
\end{eqnarray}
namely, $m=n$. Consequently,
\begin{equation}
S^{n}_{\alpha}(F(x))=\alpha^{[0]+[1]+...+[n-1]}F(x)=N(\alpha)F(x),
\end{equation}
in which $N(\cdot)=N_{\mathbb{F}_{q^n}/\mathbb{F}_q}(\cdot)$ denotes the norm function from $\mathbb{F}_{q^n}$ to $\mathbb{F}_q$. So, after applying $tn$ times the $\alpha-$cyclic shift operator $S_{\alpha} (\cdot)$ over $F(x)$, one have
\begin{eqnarray}
S^{tn}_{\alpha}(F(x))&=&\underbrace{\left(S^{n}_{\alpha} \circ \ldots \circ S^{n}_{\alpha}\right)}_{t-times}(F(x))\nonumber\\
&\equiv& N^t(\alpha)F(x)=N\left(\alpha^t \right) F(x) \nonumber\\
&\equiv& N\left(\beta^{M(q-1)}\right) F(x)\nonumber\\
&\equiv& N^M \left(\beta^{q-1}\right)F(x)\nonumber\\
&\equiv&F(x)\mod x^{[n]}-x,
\end{eqnarray}
since $N(\gamma)=1$ if and only if $\gamma=\delta^{q-1}$. 
\end{proof}

From now on, we assume that $\beta$ is a primitive element of $\mathbb{F}_{q^n}$.

Given $F(x) \in\displaystyle{\mathcal{L}_n \left(\mathbb{F}_{q^n}\right)}$ a linear permutation, notice that its $\alpha-$cyclic order is $tn$, for $1\leq t \leq q-1$, and it depends on $\alpha$ only. We say $F(x)$ has \emph{maximal $\alpha- $cyclic order} $(q-1)n$, when $\alpha=\beta^l$ so that $\gcd(l,q-1)=1$. Thus, just for simplification and into the context of this paper, we refer such $\alpha$ as maximal cyclic order element in $\mathbb{F}_{q^n}$.

Furthermore, $G(x)\in \displaystyle{\mathcal{L}_n \left(\mathbb{F}_{q^n}\right)}$ is $\alpha-$\emph{cyclically equivalent} to $F(x)$ if $S_{\alpha}^{ml}(G(x))\equiv F(x) \mod x^{[n]} -x$, for some $1\leq m\leq q-1$ and $1\leq l\leq n$.

\begin{corollary}\label{corordermax}
Given $F(x)\in \displaystyle{\mathcal{L}_n \left(\mathbb{F}_{q^n}\right)}$ a linear permutation and $\alpha \in \mathbb{F}^{*}_{q^n}$ a maximal cyclic order element, then the $\alpha-$cyclic equivalence class of $F(x)$ yields $(q-1)n$ distinct linear permutations in $\displaystyle{\mathcal{L}_n \left(\mathbb{F}_{q^n}\right)}$. \end{corollary}

\begin{proposition}\label{inverdescic}
Let $F(x)\in\mathcal{L}_n \left(\mathbb{F}_{q^n}\right)$ be a linear permutation, in which $\gcd(n,q-1)=1$. If its $t-$th $\alpha-$cyclic shift is $G(x)\equiv S^{t}_{\alpha}(F(x))$, for $\alpha \in \mathbb{F}^{*}_{q}$ that is a primitive element, then the corresponding compositional inverse is $G^{-1} (x)\equiv S^{(q-1)n - t}_{\alpha}(F^{-1}(x))$.
\end{proposition}

\begin{proof}
Calling $T(x)=S^{(q-1)n - t}_{\alpha}(F^{-1}(x))$ and $N=(q-1)n$, we have
\begin{equation}\label{eq11}
G(T(x))\equiv
F(x)\circ \underbrace{\alpha x^{[1]}\ldots \circ \alpha x^{[1]}}_{t-times}\circ F^{-1}(x)\circ \underbrace{\alpha x^{[1]}\ldots \alpha x^{[1]}}_{N-t-times}
\end{equation}
Since $\alpha x^{[1]}\in Z\left(\mathcal{L}_n \left(\mathbb{F}_{q^n}\right)\right)$, the center of the ring $\mathcal{L}_n \left(\mathbb{F}_{q^n}\right)$, the Equation~\eqref{eq11} can be rewritten as 
\begin{eqnarray}
G(T(x))&\equiv& F(x)\circ F^{-1}(x)\circ \underbrace{\alpha x^{[1]}\circ \ldots \circ \alpha x^{[1]}}_{(q-1)n - times}\nonumber\\
&\equiv& F(x)\circ F^{-1} (x)\nonumber\\
&\equiv& x \mod x^{[n]}-x
\end{eqnarray}
once $F^{-1}(x)\circ \underbrace{\alpha x^{[1]}\circ \ldots \circ \alpha x^{[1]}}_{(q-1)n - times}\equiv S_{\alpha} ^{(q-1)n}\left(F^{-1}(x)\right)\equiv F^{-1} (x)$ by Corollary~\ref{corordermax}. Thus, indeed, $T(x)\equiv G^{-1} (x)\mod x^{[n]}-x$.
\end{proof}

As a simple but important observation used on next Section, if $F(x)\in\mathcal{L}_n \left(\mathbb{F}_{q^n}\right)$ is a linear permutation, then the compositional inverse of $S(F(x))$ is $S^{n-1}(F^{-1}(x))$. 

\begin{corollary}
In addition of the hypothesis of the Proposition~\ref{inverdescic}, let $F(x)\in\mathcal{L}_n \left(\mathbb{F}_{q^n}\right)$ be an linear involution, in which $2|(q-1)n$. Then $S^{\frac{(q-1)n}{2}}_{\alpha}(F(x))$ is also an linear involution in $\mathcal{L}_n \left(\mathbb{F}_{q^n}\right)$.
\end{corollary}

\begin{proof}
It is clear from the Equation~\eqref{eq11} and $\displaystyle{t=\frac{(q-1)n}{2}}$.
\end{proof}

According to this section, it is possible to obtain several linear permutations in a trivial way; just using the $\alpha-$cyclic shifts of a given linear permutation. This construction is effective for the results proposed on next section, since it will be needed to use conventional $q-$associates of known linear permutations in order to get others, this time in a non-trivial way. The same construction can be applied in to order to yield their respective compositional inverses and, once again, the $\alpha-$cyclic shifts may develop an important role. In particular, these ideas may be applied to involution constructions.
\section{A Linear Permutation Construction Based on Primitive Idempotents}\label{construcao}
Since linear polynomials in $\mathcal{L}_n \left(\mathbb{F}_{q^n}\right)$ can be seen as linear operators over $\mathbb{F}_{q^n}$, notably it is known they are linear permutations if and only if  their kernels are trivial. In the particular case in which their coefficients are in $\mathbb{F}_{q}$, by Theorem~\ref{divisao} this is equivalent to state their respective $q-$conventional associates and $x^{n}-1\in \mathbb{F}_q [x]$ are coprimes. This well-known consideration will be taken during all this section.

It is worth recalling (See page 3) $\displaystyle{R_{q,n}:=\frac{\mathbb{F}_q [x]}{\left\langle x^n -1 \right\rangle}}$, in which $\gcd(q,n)=1$.

Next, we present the main result of this work.

\begin{theorem}\label{teoppidemp}
Let $F(x)\in \mathcal{L}_n \left(\mathbb{F}_{q}\right)$, $f(x)$ be its corresponding conventional $q-$associate and $E:=\left\{e_1 (x),...,e_t (x)\right\}$ the family of the primitive idempotents in $R_{q,n}$. $F(x)$ is a linear permutation over $\mathbb{F}_{q^n}$ if and only if 
\begin{equation}
f(x)e_i (x) \not\equiv 0 \mod x^n -1,
\end{equation}
for all $1\leq i \leq t$.
\end{theorem}
\begin{proof}
Let $f(x)$ be the conventional $q-$associate of $F(x)$. According to the discussion above, showing that $F(x)$ is a linear permutation over $\mathbb{F}_{q^n}$ is equivalent to show that\linebreak $\gcd\left(f(x), x^n -1\right)=1$. 

Suppose $F(x)$ is not a linear permutation over $\mathbb{F}_{q^n}$, namely, there exists $1\neq g(x)\in R_{q,n}$ so that $\gcd\left(f(x), x^n -1\right)= g(x)$. From Theorem~\ref{idempcodc}, there also exists an unique idempotent $1\neq e(x)\in R_{q,n}$ so that $f(x) \in \left\langle g(x) \right\rangle=\left\langle e(x) \right\rangle$.  Without loss of generality, $e(x)=e_1 (x) + e_2 (x)+...+e_l (x)$, $1\leq l<t$, is written as a sum of primitive idempotents. Since
\begin{equation}
f(x)\equiv h(x)e(x)\mod x^n -1 , 
\end{equation}
for some $h(x) \in R_{q,n}$, then
\begin{equation}
f(x)e_t (x)\equiv h(x)e(x)e_t (x) \equiv 0\mod x^n -1, 
\end{equation}
(See~\eqref{eq7}) contradicting the hypothesis.

Conversely, suppose there is a primitive idempotent $e_i (x) \in E$ so that $f(x)e_i (x)\equiv 0 \mod x^n -1.$ Since $(f(x),x^n -1)=1$ and $x^n -1 |f(x)e_i (x)$, then $x^n -1| e_i (x)$, which is impossible, since the degree of $e_i (x)$ is less than $n$ and $e_i (x)\not\equiv 0$.
\end{proof}

The following Corollary provides a simple criterion when a linear polynomial is not a linear permutation.

\begin{corollary}\label{corpp}
Let $F(x)\in \mathcal{L}_n \left(\mathbb{F}_q \right)$. If the sum of the coefficients of $F(x)$ is equivalent to $0$ in $\mathbb{F}_q$, then $F(x)$ is not a linear permutation over $\mathbb{F}_{q^{n}}$.
\end{corollary}
\begin{proof}
By Theorem~\ref{teoppidemp}, $F(x)$ is not a linear permutation, since there is at least one primitive idempotent $e_i (x)$ of $R_{q,n}$ so that the product $f(x)e_i (x)\equiv 0 \mod x^n -1$. Consider the polynomial $\displaystyle{e(x)=\frac{1}{n}\sum_{i=0}^{n-1} x^i}$. It is well-known~\cite[Lemma 3.6.6]{polcino} that $e(x)$ is an idempotent in $R_{q,n}$, which proof we reproduce here, just for completeness 
\begin{eqnarray}
e(x)e(x)&\equiv&\left(\frac{1}{n}\sum_{i=0}^{n-1} x^i\right)\left(\frac{1}{n}\sum_{i=0}^{n-1} x^i\right)\nonumber\\
&\equiv&\frac{1}{n^2}\sum_{i=0}^{n-1}x^i \left(\sum_{i=0}^{n-1} x^i\right) \nonumber\\
&\equiv& \frac{1}{n^2}\sum_{i=0}^{n-1} n x^i  \nonumber \\
&\equiv& \frac{1}{n}\sum_{i=0}^{n-1}  x^i \nonumber\\
&\equiv& e(x) \mod x^n -1.
\end{eqnarray}
Further, since~\cite[Proposition 3.6.7]{polcino} $\left\langle e(x)\right\rangle\cong \mathbb{F}_q\left(G/G\right)=\mathbb{F}_q$, then $\left\langle e(x)\right\rangle$ is a simple component and, consequently, $e(x)$ is a primitive idempotent in $R_{q,n}$. Hence, $\displaystyle{F(x)=\sum_{i=0}^{n-1}f_i x^{[i]}}$ is not a linear permutation over $\mathbb{F}_{q^n}$ if $f(x)e(x) \equiv 0 \mod x^n -1$, i.e., 
\begin{eqnarray}
\left(\sum_{i=0}^{n-1}f_i x^{i}\right)\left(\frac{1}{n}\sum_{i=0}^{n-1} x^i\right)&\equiv& \sum_{j=0}^{n-1} \left(\sum_{i=0}^{n-1} f_i\right) x^{j}\equiv 0 \mod x^n -1, \nonumber
\end{eqnarray}
namely, 
\begin{equation}
\sum_{i=0}^{n-1} f_i \equiv 0 \mod q     ,
\end{equation}
and the result follows.
\end{proof}

Based on Theorem~\ref{teoppidemp}, we provide explicit families of linear permutations over $\mathbb{F}_{q^n}$, in which $q$ and $n$ must satisfy some prescribed conditions. In this work, linear permutations are obtained using the construction of primitive idempotents in group algebras given in~\cite{ferraz}, which 
uses the group structure in order to get them in an uncomplicated way. Before presenting it, we take into consideration the natural $\mathbb{F}_q-$ algebra isomorphism between the group algebra $\mathbb{F}_q C$ and $R_{q,n}$, in which $C= \langle c \rangle$ is a $n$-order cyclic group generated by $c$, i.e.
\begin{equation}\label{isomorphismo}
    \begin{array}{ccc}
        \varphi:\mathbb{F}_q C &\rightarrow &R_{q,n}  \\
        f_0 +...+f_{n-1}c^{n-1}&\mapsto &  f_0+...+f_{n-1}x^{n-1}.
    \end{array}
\end{equation}

On next lemma, given $C_i$ a subgroup of $C$, define $\displaystyle{\widehat{C_i}=\frac{1}{\left| C_i \right|}\sum_{c \in C_i} c \in \mathbb{F}_q C}$. The notation of such lemma will be slightly altered in order to match with that one used in this work, that is, we adapt it to polynomial context.

\begin{lemma}\cite[Lemma 3]{ferraz}\label{idempotentes}
Let $\mathbb{F}_q$ be a finite field, let $p$ be a rational prime and let $C=\langle c \rangle$ be a cyclic group of order $p^m$, $m\geq 1$. Let
\begin{equation}
    C=C_0 \supseteq C_1 \supseteq C_m =\{1\}
\end{equation}
be the descending chain of all subgroups of $C$. Then the elements 
\begin{equation}
e_0 =\widehat{C}\mbox{ and } e_i = \widehat{C_i} -\widehat{C_{i-1}}, 1\leq i\leq m
\end{equation}
form a set of orthogonal idempotents of $\mathbb{F}_q C$ so that $e_0 +e_1 +...+e_m =1$.
\end{lemma}

On next Corollary, $U\left(\mathbb{Z}_{p^m}\right)$, $o(\cdot)$ and $\phi(\cdot)$ mean the subgroup of units in $\mathbb{Z}_{p^m}$, the multiplicative order of an element in $U\left(\mathbb{Z}_{p^m}\right)$ and the classical Euler's Totient function, respectively. Once more, we slightly altered the writing of this Corollary in order to preserve the same notation throughout the work.

\begin{corollary}\cite[Corollary 4]{ferraz}\label{corolarioidemp}
Let $\mathbb{F}_q$ be a finite field, and let $C$ be a cyclic group of order $p^m$.
Then, the set of idempotents given in Lemma~\ref{idempotentes} is the set of primitive idempotents of $\mathbb{F}_q C$ if
and only if one of the following holds:
\begin{itemize}
\item[(i)] $p = 2$, and either $m = 1$ and $q$ is odd or $m = 2$ and $q \equiv 3 \mod 4$ or
\item[(ii)]  $p$ is an odd prime and $o(q) = \phi\left(p^m \right )$ in $U\left(\mathbb{Z}_{p^m}\right)$.
\end{itemize}
\end{corollary}

According to~\cite{ferraz}, it is also possible to describe all primitive idempotents of $\mathbb{F}_q C$, in which $C$ is $2p^m-$order cyclic group. See~\cite[Theorem 3.2]{ferraz}.

\begin{remark}
Based on isomorphism~\eqref{isomorphismo}, we will adopt the polynomial notation in order to present the primitive idempotents of $R_{q,p^m}$, consequently, families of linear permutations over $\mathbb{F}_{q^{p^m}}$. 

It is important to stress that there are more general constructions of primitive idempotents which may be adapted to this work, providing linear permutations in $\mathcal{L}_n \left(\mathbb{F}_q\right)$, in which $n$ is not restricted to $p^n$ or $2p^n$.
\end{remark}

\begin{example}\label{exemplo}
Let $F(x)=f_0x + f_{25} x^{[25]} + f_{124} x^{[124]} \in \mathcal{L}_{125}\left(\mathbb{F}_3 \right)$. Since $o(3)=100=\phi(125)$ in $U\left(\mathbb{Z}_{125} \right)$, then the primitive idempotents of $R_{3,125}$ are
\begin{itemize}
\item[(i)] $\displaystyle{e_0 (x) =\frac{1}{125}\sum_{i=0}^{124}x^i\equiv 2\sum_{i=0}^{124}x^i}$
\item[(ii)]$\displaystyle{e_1 (x) =\frac{1}{25}\sum_{i=0}^{24}x^{5i}-\frac{1}{125}\sum_{i=0}^{124}x^i\equiv \sum_{i=0}^{24}x^{5i}+\sum_{i=0}^{124}x^i}$
\item[(iii)]$\displaystyle{e_2 (x) =\frac{1}{5}\sum_{i=0}^{4}x^{25i}-\frac{1}{25}\sum_{i=0}^{24}x^{5i}\equiv 2\sum_{i=0}^{4}x^{25i}+\sum_{i=0}^{24}x^{5i}}$
\item[(iv)]$\displaystyle{e_3 (x) =1-\frac{1}{5}\sum_{i=0}^{4}x^{25i}\equiv 1+\sum_{i=0}^{4}x^{25i}}$.
\end{itemize}

Taking the conventional $3-$associate of $F(x)$, $f(x)=f_0 + f_{25} x^{25} + f_{124} x^{124}$, and based on the products $f(x)e_i (x)\not\equiv 0\mod x^{125} -1$, for all $0\leq i \leq 3$, the Table~\ref{tab:tablei} describes all respective linear permutations (linearized $3-$associates) over $\mathbb{F}_{3^{125}}$ which coefficients are in $\mathbb{F}_3$. Since one of the monomials and each linear permutation from the first column are not cyclically equivalent to each other, it is still possible to apply the $S^n (\cdot)-$operator in each one in order to get more $125$ linear permutations (repeating the other two monomials). 
%

\begin{table}[h!]
\centering
\caption{Some Linear Permutation over $\mathbb{F}_{3^{125}}$\label{tab:tablei}}
\begin{tabular}{c|c}
\hline
$x$ & $2x$ \\ \hline
$x^{[25]}$ &  $2x^{[25]}$\\ \hline
$x^{[124]}$&$2x^{[124]}$ \\ \hline
 $x^{[25]}+x$ & $2x^{[25]}+2x$ \\ \hline
$x^{[124]}+x$    & $2x^{[124]}+2x$  \\ \hline
$x^{[124]} + x^{[25]}$&$2x^{[124]} + 2x^{[25]}$   \\ \hline
$2x^{[124]}+x^{[25]}+x$ &$x^{[124]}+2x^{[25]}+2x$ \\ \hline
$x^{[124]}+x^{[25]}+2x$& $2x^{[124]}+2x^{[25]}+x$
\\ \hline
$x^{[124]}+2x^{[25]}+x$ & $2x^{[124]}+x^{[25]}+2x$ \\ \hline
\end{tabular}
\end{table}
\end{example}
%

\begin{corollary}\label{corpp1}
Let $\mathbb{F}_q$ be a finite field and $p$ an rational odd prime so that $o(q) = \phi\left(p^m \right )$ in $U\left(\mathbb{Z}_{p^m}\right)$. Given $\displaystyle{F(x)=\sum_{i=0}^{n-1} f_i x^{[i]} \in \mathcal{L}_{p^m}\left(\mathbb{F}_q\right)}$, if
\begin{itemize}
\item [(i)] $\displaystyle{\sum_{j=0}^{p^m -1}f_j\not\equiv 0}$ and \item[(ii)]
$\displaystyle{-p^{-m+i-1}\sum_{\substack{j=0,\\ p \nmid j}}^{p^{m-i+1} -1} f_{p^m -jp^{i-1}}+\left(p^{-m+i}-p^{-m+i-1}\right)\sum_{\substack{j=0,\\ p \mid j}}^{p^{m-i+1} -1} f_{p^m -jp^{i-1}}\not\equiv 0,}$ 
\end{itemize}
over $\mathbb{F}_q$, for all $1\leq i\leq m$ in which all subscripts are taken$\mod p^m$, then $F(x)$ is a linear permutation over $\mathbb{F}_{q^{p^m}}$.
\end{corollary}

\begin{proof}
According to Theorem~\ref{teoppidemp}, we should demonstrate that the product of the conventional $q-$associate of $F(x)$, $f(x)$, with all primitive idempotents of $R_{q,p^m }$ are not equivalent to zero. From Corollary~\ref{corolarioidemp}-(ii), the idempotents
\begin{eqnarray}
e_0 (x) &=& \widehat{A_0}=p^{-m}\left(1+x+x^2+...+x^{p^m -1}\right)\mbox{ and}\nonumber\\
e_i (x)&=&\widehat{A_i}-\widehat{A_{i-1}}\nonumber\\
&=&p^{-m+i}\left(1+x^{p^i} +...+x^{p^i(p^{m-i} -1)}\right)\nonumber\\
&-&p^{-m+i-1}\left(1+x^{p^{i-1}} +...+x^{p^{i-1}(p^{m-i+1} -1)}\right)\nonumber\\
&=&\sum_{\substack{j=0,\\ p \nmid j}}^{p^{m-i+1} -1} -p^{-m+i-1}x^{jp^{i-1}}+\sum_{\substack{j=0,\\ p\mid j}}^{p^{m-i+1} -1} \left(p^{-m+i}-p^{-m+i-1}\right)x^{jp^{i-1}},  
\end{eqnarray}
$1\leq i \leq m$, correspond to all primitive idempotents of $R_{q,p^m }$.

A trivial way to ensure that the products $f(x)e_i (x)$ are $\not\equiv 0 \mod x^{p^m} -1$, for all $0\leq i\leq m$, it is, for instance, to guarantee that each independent term of every product is non-zero over $\mathbb{F}_q$. Defining $\displaystyle{g_i(x)=\sum_{j=0}^{n-1} g_{i,j}x^j}$ and $\displaystyle{e_i(x)=\sum_{j=0}^{n-1} e_{i,j}x^j}$, $0\leq i\leq m$, so that $g_i(x)\equiv f(x)e_i (x)\mod x^{p^m} -1$, we have
\begin{eqnarray}
g_{0,0}&\equiv& \sum_{j=0}^{p^m -1} f_{p^m-j}e_{0,j}\Rightarrow \sum_{j=0}^{p^m -1}f_j\not\equiv 0 \mbox{ and} \nonumber\\
g_{i,0}&\equiv&\sum_{j=0}^{p^{m-i+1} -1}f_{p^m -jp^{i-1} }e_{i,jp^{i-1}}\nonumber\\
&\equiv& \sum_{\substack{j= 0,\\ p \nmid j}}^{p^{m-i+1} -1}f_{p^m -jp^{i-1} }e_{i,jp^{i-1}}+\sum_{\substack{j=0,\\ p \mid j}}^{p^{m-i+1} -1} f_{p^m -jp^{i-1} }e_{i,jp^{i-1}}\nonumber\\
&\equiv& -p^{-m+i-1}\sum_{\substack{j=0,\\ p \nmid j}}^{p^{m-i+1} -1} f_{p^m -jp^{i-1}}+\left(p^{-m+i}-p^{-m+i-1}\right)\sum_{\substack{j=0,\\ p \mid j}}^{p^{m-i+1} -1} f_{p^m -jp^{i-1}}\nonumber\\
&\not\equiv&0,
\end{eqnarray}
over $\mathbb{F}_q$, for each $1\leq i \leq m$, and all these subscripts into the summations are taken $\mod p^m$.
\end{proof}

\begin{remark}
It is worth mentioning that the proof of the Corollary~\ref{corpp1} has been done considering only coefficients from the independent terms of the polynomials $g_i(x)$, for $1\leq i \leq m$. However, it could be done choosing any other coefficients from these polynomials.
\end{remark}

\begin{corollary}\label{binppp}
Let $F(x)=f_i x^{[i]} + f_j x^{[j]} \in \mathcal{L}_n \left(\mathbb{F}^{*}_q \right)$, in which $0\leq i < j\leq n-1$. $F(x)$ is a linear permutation over $\mathbb{F}_{q^n}$ if and only if $f_i + f_j \not\equiv 0$ over $\mathbb{F}_q$.
\end{corollary}

\begin{proof}
If $F(x)$ is a linear permutation, then $f_i + f_j \not\equiv 0$ in $\mathbb{F}_q$, according to Corollary~\ref{corpp}.

Conversely, let $E=\left\{e_1 (x),...,e_t (x)\right\}$ be the set of primitive idempotents of $R_{q,n}$. In order to ensure that $F(x)$ is a linear permutation over $\mathbb{F}_{q^n}$, beyond the condition $f_i + f_j \not\equiv 0$, we have other $t-1$ conditions $\alpha_l f_i + \beta_l f_j =\delta_ l \not\equiv 0$, each one seen as a coefficient of the products $f(x)e_l(x) \mod x^n -1$, in which $2\leq l \leq t$ and $\alpha_l, \beta_l, \delta_ l \in \mathbb{F}_q$. Thus, we have the following equation system

\begin{equation}\label{sistema}
\left\{\begin{array}{ccc}
    f_i + f_j & =&\delta_1 \\
    \alpha_2 f_i + \beta_2 f_j & =&\delta_2 \\
    &\vdots&\\
    \alpha_t f_i + \beta_t f_j & =&\delta_t 
\end{array}\right..
\end{equation}

Applying the Gaussian Elimination on~\eqref{sistema} when needed, this equation system is reduced to the first equation and the other equations are simply descriptions of the non-zero element $f_i$ (or $f_j$) based on elements of $\mathbb{F}_q$. Therefore, the result follows.
\end{proof}


\begin{example}\label{ex22}
Still working on $R_{3,125}$ as in Example~\ref{exemplo}, consider $F(x)=f_{64}x^{[64]}+f_{63}x^{[63]}+f_{62}x^{[62]}+f_0x \in \mathcal{L}_{125}\left(\mathbb{F}_3 \right)$. Based on Corollary~\ref{corpp}, such linear polynomial is a linear permutation over $\mathbb{F}_{3^{125}}$ if 
\begin{equation}\label{rest}
\left\{\begin{array}{ccccccccc}
2f_0&+&2f_{62}&+&2f_{63}&+&2f_{64}&\not\equiv&0  \\
2f_0&+&f_{62}&+&f_{63}&+&f_{64}&\not\equiv&0  \\
f_0&&&&&&&\not\equiv&0  \\
2f_0&&&&&&&\not\equiv&0  
\end{array}\right.    
\end{equation}
over $\mathbb{F}_3$. Hence, 
\begin{eqnarray*}
F_1 (x)&=&x^{[64]}+x^{[63]}+x^{[62]}+2x,\\
F_2 (x)&=&2x^{[64]}+ 2x^{[63]}+ 2x^{[62]}+2x,\\
F_3 (x)&=&2x^{[64]}+ x^{[63]}+ x\mbox{ and }\\
F_4 (x)&=&x^{[64]}+ 2x^{[63]}+ x 
\end{eqnarray*}
are some examples of linear permutations over $\mathbb{F}_{3^{125}}$ which satisfy~\eqref{rest}. Moreover, observe that they are non-$\alpha-$cyclically equivalent, for $\alpha \in \mathbb{F}^{*}_{3^{125}}$.
\end{example}

Note that the restrictions imposed by Corollary~\ref{corpp1} might be simplified. On Example~\ref{ex22}/Equation~\eqref{rest}, it could be reduced only to $F(x)\in \mathcal{L}_{125}\left(\mathbb{F}_{3}\right)$ so that $f_0 + f_{62}+ f_{63}+ f_{64}\not\equiv 0$ and $f_0\neq 0$.

The Corollary~\ref{corpp1} is useful to describe some $\lambda-$complete linear permutations~\cite{bcpp} over $\mathbb{F}_{q^{p^m}}$, for $\lambda \in \mathbb{F}_q$. In fact, it inspires a more general definition (defined in~\cite{tuxanidy}, but not named like that) which we call $A-$complete linear permutations, in which $A\subset \mathbb{F}_q$. 
\begin{definition}\label{aclp}
Let $f(x)$ be a permutation polynomial over $\mathbb{F}_{q^n}$. Given $A\subset \mathbb{F}_{q^n}$, we define $f(x)$ is a \emph{$A$-complete permutation polynomial} over $\mathbb{F}_{q^n}$ if $f(x)+\lambda x$ is also a permutation polynomial over $\mathbb{F}_{q^n}$, for all $\lambda \in A$.
\end{definition}

From the definition above, it is straightforward $A\neq \emptyset$, since $0 \in A$.

It is possible to adapt the Corollary~\ref{corpp1} in order to offer conditions on $F(x)$ so that it is $A-$complete linear permutation over $\mathbb{F}_{q^{p^m}}$. 

\begin{corollary}\label{corpp2}
Let $A=\left\{0,\lambda_2, ...,\lambda_m \right\}\subset \mathbb{F}_q$ and $p$ an rational odd prime so that $o(q) = \phi\left(p^m \right )$ in $U\left(\mathbb{Z}_{p^m}\right)$. Given $\displaystyle{F(x)=\sum_{i=0}^{n-1} f_i x^{[i]} \in \mathcal{L}_{p^m}\left(\mathbb{F}_q\right)}$, if
\begin{itemize}
\item [(i)] $\displaystyle{\sum_{j=0}^{p^m -1}f_j\not\equiv -\lambda_l}$ and \item[(ii)]
$\displaystyle{-p^{-m+i-1}\sum_{\substack{j=0,\\ p \nmid j}}^{p^{m-i+1} -1} f_{p^m -jp^{i-1}}+\left(p^{-m+i}-p^{-m+i-1}\right)\sum_{\substack{j=0,\\ p \mid j}}^{p^{m-i+1} -1} f_{p^m -jp^{i-1}}}$\\
$\not\equiv -\left(p^{-m+i}-p^{-m+i-1}\right)\lambda_l,$ 
\end{itemize}
over $\mathbb{F}_q$, for all $1\leq i\leq m$ and all $\lambda_l \in A$, then $F(x)$ is a $A-$complete linear permutation over $\mathbb{F}_{q^{p^m}}$.
\end{corollary}

\begin{example}
We know $F(x)=f_t x^{[t]}\in \mathcal{L}_{11} \left(\mathbb{F}_8 \right)$ (See~\cite[Theorem 7.8, (ii)]{finitefields}) is a linear permutation over $\mathbb{F}_{8^{11}}$, for any $0 \leq t \leq 10$. Since $o(8)=10=\phi(11)$ in $U\left(\mathbb{Z}_{11} \right)$, so the primitive idempotents of $R_{8,11}$ are 
\begin{equation}
e_0 (x)=\sum_{i=0}^{10} x^i \mbox{ and }
e_1 (x)=1-e_0 (x). 
\end{equation}

From Corollary~\ref{binppp}, we observe that $F(x)+\lambda x$ is also a linear permutation over $\mathbb{F}_{8^{11}}$ if and only if $\lambda\neq -f_t$, namely, $F(x)$ is a $\mathbb{F}_8 \setminus\left\{-f_t \right\}-$complete linear permutation over $\mathbb{F}_{8^{11}}$.
\end{example}

Next, it is presented a way to construct linear permutations and their compositional inverses over $\mathbb{F}_{q^n}$ by the use of units in $R_{q,n}$; equivalently conventional $q-$associates of linear permutations in $\mathcal{L}_n \left(\mathbb{F}_q\right)$. Once more, the use of primitive idempotents comes as an essential tool, since we analyze the projections of $F(x)\in \mathcal{L}_n \left(\mathbb{F}_q\right)$ over the simple components generated by the primitive idempotents. As it will be seen, the task of getting units on $R_{q,n}$ may be easily simplified using some constructions of Section~\ref{sectioni}. 

\begin{theorem}\label{tp}
Let $E=\left\{e_1 (x),..., e_t (x)\right\}$ be the set of primitive idempotents of $R_{q,n}$. Given $F(x)\in \mathcal{L}_{n}\left(\mathbb{F}_q\right)$ a linear  permutation over $\mathbb{F}_{q^n}$ and $f(x)$ its conventional $q-$associate, then $f(x)$ can be written as
\begin{equation}\label{soma}
f(x)\equiv \sum_{i=1}^t f_i (x) e_i (x) \mod x^n -1, 
\end{equation}
in which $\gcd\left(f_i (x) , x^n -1\right)=1$, for $1\leq i\leq t$. Furthermore, the conventional $q-$associate of the compositional inverse of $F(x)$ is given as
\begin{equation}
f^{-1} (x)\equiv   \sum_{i=1}^t f_i^{-1} (x) e_i (x) \mod x^n -1.   \end{equation}

Every polynomial $f(x)$ (respectively $f^{-1} (x)$) is uniquely determined by vector \linebreak$(f_1 (x),..., f_t (x))$ (respectively $(f^{-1}_{1} (x),..., f^{-1}_{t} (x))$).
\end{theorem}

\begin{proof}
Given $F(x)\in \mathcal{L}_n \left(\mathbb{F}_q\right)$ a linear permutation over $\mathbb{F}_{q^n}$, from Theorem~\ref{teoppidemp} we have $f(x)e_i (x)\not\equiv 0 \mod x^n -1$, for all $1\leq i \leq t$. We may assume that $
f(x)e_i(x)\equiv f_i(x)e_i(x) \mod x^n -1$ and $\gcd\left(f_i (x), x^n -1\right)=1$, for all $1\leq i \leq t$, since $\left\langle e_i (x) \right\rangle$ corresponds to simple component/finite field in $R_{q,n}$. Indeed, if $\gcd\left(f_i (x), x^n -1\right) =g(x)\neq 1$, then there is an idempotent $e(x)=e_{i_1} (x)+...+e_{i_j}(x)$, $I:=\left\{i_1 ,..., i_j \right\}\subset \{1,..,t\}$, so that $\langle g(x) \rangle= \langle e(x) \rangle$, and
\begin{equation}
f_i (x)\equiv h_i (x) g(x) \equiv \overline{h_i(x)} e(x)\mod x^n -1
\end{equation}
in which $h_i(x), \overline{h_i(x)}\in R_{q,n}$ and $\gcd\left(h_i (x), x^n -1\right)=\gcd\left(\overline{h_i (x)}, x^n -1\right)=1$. Thus
\begin{equation}
f_i (x) e_i(x)\equiv \overline{h_i(x)} e(x) e_i(x) \equiv \left\{\begin{array}{l}
    0,  \mbox{ if } i\not\in I  \\
    \\
     \overline{h_i (x)} e_i (x),\mbox{ otherwise }    
\end{array} \right. .
\end{equation}

Notice the first case above contradicts $f(x)$ being unit on $R_{q,n}$. Hence, without loss of generality, we may assume $\gcd\left(f_i (x), x^n -1\right)=1$. 

So, since $\displaystyle{\sum_{i=1}^t e_i (x)= 1}$ (See Equation~\eqref{soma1}) and from
\begin{eqnarray}\label{eqequiv}
f(x)e_1(x)&\equiv& f_1 (x) e_1 (x)\mod x^n -1  \nonumber \\
&\vdots&\\
f(x)e_t(x)&\equiv& f_t (x) e_t (x)\mod x^n -1 \nonumber,  
\end{eqnarray}
we have $\displaystyle{f(x)\equiv \sum_{i=1}^t f_ i (x) e_i (x)\mod x^n -1.}$ In particular, since $\gcd\left(f_i(x), x^n -1\right)=1$, each equivalence in~\eqref{eqequiv} can be rewritten as $f(x)f_i ^{-1} (x) e_i (x)\equiv e_i (x)\mod x^n -1$, for all $1\leq i \leq t$, consequently, we have
\begin{equation}
f(x)\sum_{i=1} ^t f_i ^{-1} (x)e_i(x)\equiv \sum_{i=1}^t e_i (x)\equiv 1 \mod x^n -1,
\end{equation}
therefore, $\displaystyle{f^{-1} (x)\equiv \sum_{i=1} ^t f_i ^{-1} (x)e_i(x) \mod x^n -1}$.

Finally, it is clear to observe that the decomposition of $F(x)$ over the simple components of $R_{q,n}$ is unique. Indeed, by the isomorphism~\eqref{tcr}, given $G(x), F(x)\in \mathcal{L}_n \left(\mathbb{F}_{q}\right)$, their corresponding conventional $q-$associates $f(x)$ and $g(x)$ are equal if and only if their projections over the simple components are all equal, namely, $f_i (x) e_i(x) \equiv g_i (x) e_i(x) \mod x^n -1$, for all $1\leq i \leq t$. 
\end{proof}

From known constructions for linear permutations in $\mathcal{L}_n \left(\mathbb{F}_q\right)$, that is, from simplest ones as monomials $\lambda x$, $\lambda \in \mathbb{F}^{*}_q$, to linear permutations of the type $x+x^2 +Tr_{2^n}\left(\frac{x}{a}\right)$ over $\mathbb{F}_{2^n}$~\cite{compinverse2}, in which $a\in \mathbb{F}^{*}_{2^n}$, their corresponding conventional $q-$associates may be used along with Theorem~\ref{tp} in order to provide new linear permutations.

\begin{example}
From Corollary~\ref{corolarioidemp}, the respective polynomials below are in fact primitive idempotents in $R_{3,25}$ 
\begin{eqnarray}
e_0 (x)&=&x^{24}+x^{23}+x^{22}+x^{21}+x^{20}+x^{19}+x^{18}+x^{17}\nonumber \\
&+&+x^{16}+x^{15}+x^{14}+x^{13}+x^{12}+x^{11}+x^{10}+x^9\nonumber\\
&+&x^8 +x^7 +x^6 +x^5 +x^4 +x^3 +x^2 +x+1 \nonumber\\
e_1 (x)&=&2x^{24}+2x^{23}+2x^{22}+2x^{21}+x^{20}+2x^{19}+2x^{18}\nonumber\\
&+&2x^{17}+2x^{16}+x^{15}+2x^{14}+2x^{13}+2x^{12}+2x^{11}\nonumber\\
&+&+x^{10}+2x^9 + 2x^8 +2x^7 +2x^6 +x^5+ 2x^4 \nonumber\\ 
&+&2x^3 +2x^2 + 2x+1 \mbox{ and}\nonumber\\
e_2(x)&=&x^{20}+x^{15}+x^{10}+x^5 +2.
\end{eqnarray}

According to Theorem~\ref{tp}, in order to obtain a linear permutation $F(x)$ and its respective compositional inverse, it is enough to get units $f_0 (x), f_1 (x), f_2 (x) \in R_{3,25}$, once the conventional $3-$associate of $F(x)$, $f(x)$, comes from as $\displaystyle{f(x)\equiv \sum_{i=0}^2 f_ i (x) e_i (x)\mod x^{25} -1.}$ 
When we exchange the order of the polynomials $f_i (x)$ in this sum, it provides other conventional $3-$associates of linear permutations. Indeed, take 
\begin{equation}
f_0 (x)= x^{20}+2x^{15}+1,    
\end{equation}
which is a unit in $R_{3,25}$. Assume $f_1 (x)=x^{21}+2x^{16}+x$ and $f_2 (x) =x^{22}+2x^{17}+x^2$,  which are also units in $R_{3,25}$, once they are left-cyclic shifts of $f_0 (x)$ (See Proposition~\ref{deslciclico}). Since $f_0 ^{-1} (x)= 2x^{20}+2x^{10}+x^5 +2$, consequently 
\begin{eqnarray}
f_1 ^{-1} (x)&=& 2x^{24}+2x^{19}+2x^9 +x^4\equiv x^{24}f_0 ^{-1} (x)\mbox{ and }\nonumber \\
f_2 ^{-1} (x)&=& 2x^{23}+2x^{18}+2x^8 +x^3\equiv x^{23}f_0 ^{-1} (x)\nonumber , 
\end{eqnarray}
so the polynomials
\begin{equation}
f_{\sigma} (x)\equiv \sum_{i=0}^2  f_{\sigma(i)} (x) e_i (x)\mod x^{25} -1,
\end{equation}
for $\sigma \in S_3$ (symmetric group), and
\begin{equation}
f_{\sigma} ^{-1} (x)\equiv \sum_{i=0}^2  f_{\sigma(i)}^{-1} (x) e_i (x)\mod x^{25} -1    
\end{equation}
are inverses of each other. The corresponding linearized $3-$associates and their respective compositional inverses are described at Table~\ref{tab:tabelaii}.
\begin{table*}[htb!]
\centering
\caption{Some linear permutations over $\mathbb{F}_{3^{25}}$ and their compositional inverses\label{tab:tabelaii}}
\begin{tabular}{ccc}
\hline
$S_3$&$F(x)$ & $F^{-1}(x)$ \\
 \hline
\multirow{2}{*}{$Id$} & $2x^{[22]}+2x^{[21]}+2x^{[16]}+x^{[12]}+2x^{[11]}$ & $2x^{[24]}+2x^{[19]}+2x^{[14]}+x^{[13]}+2x^{[9]} $ \\
 & $+x^{[7]} +2x^{[6]} +2x^{[2]} +2x^{[1]}$ & $ +2x^{[4]} +2x^{[3]}$ \\
\hline
\multirow{2}{*}{$(012)$} & $2x^{[22]}+2x^{[20]}+2x^{[17]}+2x^{[12]}+x^{[10]}$  & $2x^{[23]}+2x^{[18]}+x^{[15]}+2x^{[13]}+2x^{[8]} $  \\
& $+2x^{[7]} +x^{[5]} +2x^{[2]} +2x$  & $+2x^{[5]} +2x^{[3]}$  \\
\hline
\multirow{2}{*}{$(021)$} & $2x^{[21]}+2x^{[20]}+2x^{[15]}+x^{[11]}+2x^{[10]}$ & $2x^{[20]}+2x^{[15]}+x^{[14]}+2x^{[10]}+2x^{[5]} $ \\
& $+x^{[6]} +2x^{[5]}+ 2x^{[1]}+2x$ & $+2x^{[4]} +2x$ \\
\hline
\multirow{2}{*}{$(02)$}& $2x^{[21]}+2x^{[20]}+2x^{[16]}+2x^{[11]}+x^{[10]}$ & $2x^{[24]}+2x^{[19]}+x^{[15]}+2x^{[14]}+2x^{[9]}$ \\
& $+2x^{[6]}+x^{[5]}+2x^{[1]}+2x$ & $+2x^{[5]}+2x^{[4]}$ \\
\hline
\multirow{2}{*}{$(12)$} & $2x^{[22]}+2x^{[21]}+2x^{[17]}+2x^{[12]}+x^{[11]}$  & $2x^{[23]}+2x^{[18]}+x^{[14]}+2x^{[13]}+2x^{[8]}$ \\
& $+2x^{[7]}+x^{[6]}+2x^{[2]}+2x^{[1]}$  & $+2x^{[4]}+2x^{[3]}$ \\
\hline
\multirow{2}{*}{$(01)$}& $2x^{[22]}+2x^{[20]}+2x^{[15]}+x^{[12]}+2x^{[10]}$ & $2x^{[20]}+2x^{[15]}+x^{[13]}+2x^{[10]}+2x^{[5]}$\\
& $+x^{[7]}+2x^{[5]}+2x^{[2]}+2x$ & $+2x^{[3]}+2x$\\
\hline
\end{tabular}
\vspace{0.1 cm}
\end{table*}
\end{example}

In particular, the Theorem~\ref{tp} is also a very useful tool in the search of involutions over $\mathbb{F}_{q^n}$.
\begin{corollary}\label{involucoes}
Let $E=\left\{e_1 (x),..., e_t (x)\right\}$ be the set of primitive idempotents of $R_{q,n}$. $F(x)\in \mathcal{L}_n\left(\mathbb{F}_q \right)$ is an linear involution over $\mathbb{F}_{q^n}$ if and only if its conventional $q-$associate  is
\begin{equation}
f(x)\equiv \sum_{i=1}^t f_i (x) e_i (x) \mod x^n -1, 
\end{equation}
in which $f_i ^2 (x)\equiv 1\mod x^n -1$, for all $1\leq i \leq t$.
\end{corollary}

Notice that $R_{q,n}$ is not a finite field; in fact it is a ring with zero divisors. Hence, the polynomial $X^n -1$ may have more than $n$ roots in $R_{q,n}[X]$. In particular, $f_i ^2 (x)-1=0$ has at least two solutions in $R_{q,n}$: $f_i (x)\in \{1,-1\} \subset \mathbb{F}_q$. With this trivial information, it becomes very simple to construct some involutions over $\mathbb{F}_{q^n}$, namely,
\begin{equation}\label{maisoumenosinv}
\pm e_0 (x)\pm e_1 (x)\pm \ldots \pm e_t (x) \mod x^n -1     
\end{equation}
describe several conventional $q-$associates of involutions over $\mathbb{F}_{q^n}$. In addition, the involutions obtained in~\eqref{maisoumenosinv} may be used to get other linear involutions over $\mathbb{F}_{q^n}$.

\begin{example}
From Corollary~\ref{corolarioidemp}, the respective polynomials below are in fact primitive idempotents in $R_{11,9}$
\begin{eqnarray*}
e_0 (x)&=& 5x^8 +5x^7 +5x^6 +5x^5 +5x^4 +5x^3 +5x^2\nonumber \\
&+& 5x+5\\
e_1 (x)&=& 6x^8 +6x^7 +10x^6 +6x^5+ 6x^4 +10x^3 +6x^2 \nonumber\\
&+&6x+10\mbox{ and }\\
e_2 (x)&=& 7x^6 +7x^3 +8.
\end{eqnarray*}

As noted before, $\displaystyle{f(x)\equiv \sum_{i=0}^2 f_i (x) e_i (x)}\mod x^{9}-1$, in which $f_i (x)\in\{1, 10\}\subset \mathbb{F}_{11}$, describe conventional $11-$associates of involutions $F(x) \in \mathcal{L}_9 \left(\mathbb{F}_{11}\right)$ over $\mathbb{F}_{11^{9}}$. The Table~\ref{tab:tabelaiii} presents all involutions obtained from this way. Such involutions may be applied again to Corollary~\ref{involucoes} to get other ones. 
\begin{table*}[h!]
\centering
\caption{Involutions over $\mathbb{F}_{11^{9}}$\label{tab:tabelaiii}}
\begin{tabular}{cc}
\hline
$f_0 (x),f_1 (x),f_2 (x)$&$F(x)$ 
\\
\hline
$1,1,1$& $x$ 
\\
$1,1,10$ & $8x^{[6]}+8x^{[3]}+7x$
\\
$1,10,10$ & $10x^{[8]}+10x^{[7]}+10x^{[6]}+10x^{[5]}+10x^{[4]}+10x^{[3]}+10x^{[2]}+10x^{[1]}+9x$  
\\
$1,10,1$& $10x^{[8]}+10x^{[7]}+2x^{[6]}+10x^{[5]}+10x^{[4]}+2x^{[3]}+10x^{[2]}+10x^{[1]}+3x$  
\\
10,1,1& $x^{[8]}+x^{[7]}+x^{[6]}+x^{[5]}+x^{[4]}+x^{[3]}+x^{[2]}+x^{[1]}+2x$ 
\\
10,10,1& $3x^{[6]}+3x^{[3]}+4x$ 
\\
10,1,10& $x^{[8]}+x^{[7]}+9x^{[6]}+x^{[5]}+x^{[4]}+9x^{[3]}+x^{[2]}+x^{[1]}+8x$ 
\\
10,10,10& $10x$ 
\\
\hline
\end{tabular}
\vspace{0.1 cm}
\end{table*}
\end{example}
\section{Conclusion}\label{conclusion}
In this work, we have presented a simple way to construct a lot linear permutations over $\mathbb{F}_{q^n}$ and their respective compositional inverses, examining the projections of the corresponding conventional $q-$associates over the simple components generated by primitive idempotents of $R_{q,n}$. Particularly, the matter of getting involutions has also been addressed in here.

Such construction is simple and very effective since, for instance, in possession of a only pair $\left\{f(x),f^{-1}(x)\right\}$ in $R_{q,n}$ and their left-cyclic shifts, it is possible to provide several linear permutations and their compositional inverses.

As a future problem, it will be interesting to analyze how many fixed points these linear permutations/involutions have got and their behaviours, once this is a current research problem due to applications on cryptography.

\section*{Acknowledgment}

The authors would like to thank...

\ifCLASSOPTIONcaptionsoff
\fi




\begin{thebibliography}{1}

\bibitem{agw} A. Akbary, D. Ghioca and Q. Wang, \emph{On Constructiong Permutations on Finite Fields}, Finite Fields and Their Applications, vol 17, 51-67, 2011.

\bibitem{bcpp} L.A. Bassalygo and V.A. Zinoviev, \emph{
Permutation and Complete Permutation Polynomials},
Finite Fields and Their Applications,
vol 33, 198-211, 2015.

\bibitem{prince} J. Borghoff et al., \emph{PRINCE - A Low-Latency Block Cipher for Pervasive Computing Applications}, in Advances in Cryptology, vol 7658, 208-225, 2012.

\bibitem{charpin} P. Charpin, S. Mesnager and S. Sarkar, \emph{Involutions Over the Galois Field $ {\mathbb F}_{2^{n}}$}, IEEE Transactions on Information Theory, vol 62, 2266-2276, 2016.

\bibitem{polcycliccodes} C. Ding and S. Ling, \emph{A q-Polynomial Approach to Cyclic Codes}, Finite Fields and Their Applications, vol 20, 1-14, 2013.

\bibitem{ferraz} R. Ferraz and C. Polcino Milies, \emph {Idempotents in Group Algebras and Minimal Abelian
Codes}, Finite Fields and Their Applications, vol 13, 382-393, 2007.

\bibitem{gilbert} E. N. Gilbert, \emph{Cyclically Permutable Error-Correcting Codes}, IEEE Transactions on Information Theory, vol 9, 175-182, 1963.

\bibitem{grosso} V. Grosso et al., \emph{SCREAM \& iSCREAM Side-Channel Resistant Authenticated Encryption With Masking}, CAESAR, Available:
https://competitions.cr.yp.to/round1/screamv1.pdf, 2014.

\bibitem{koetterk} R. K\"otter and F.R. Kschischang, \emph{Coding for Errors and Erasures in Random Network Coding}, IEEE Transactions on Information Theory, vol 54, 3579-3591, 2008.

\bibitem{invmon} G. M. Kyureghyan and V. Suder, \emph{On Inversion in $\mathbb{Z}_{2^n -1}$}, Finite Fields and Their Applications, vol 25, 234-254, 2014.

\bibitem{app} J. Levine and J. V. Brawley, \emph{Some Cryptographic Applications of Permutation Polynomials}, Cryptologia, vol 1, 76-92, 1977. 

\bibitem{finitefields} R. Lidl and H. Niederreiter, \emph{Finite Fields}, Cambridge University Press, 1997.

\bibitem{linearpoldimension} S. Ling and L. Qu, \emph{A Note on Linearized Polynomials and the Dimension of Their
Kernels}, Finite Fields and Their Applications, vol 18, 56-62, 2012.

\bibitem{mac} F. J. MacWilliams and N. J. A. Sloane, \emph{The Theory of Error-Correcting Codes}, North Holand, 1983.

\bibitem{completemaps}
H. Niederreiter and K. Robinson, \emph{Complete Mappings of Finite Fields}, Journal of the Australian Mathematical Society, vol 33, 197-212, 1982.

\bibitem{niu} T. Niu, K. Li, L. Qu and Q. Wang, \emph{New Constructions of Involutions over Finite Fields}, Cryptography and Communications, vol 12, 165-185, 2020.

\bibitem{polcino} C. Policino Milies and S.K. Sehgal, \emph{An Introduction to Group Rings}, Kluwer Academic Publishers, Dordrecht, 2002.

\bibitem{lucas} L. Reis, \emph{Nilpotent Linearized Polynomials over Finite Fields
and Applications}, Finite Fields and Their Applications,
vol 50, 279-292, 2018.

\bibitem{valdemar} V. C. da Rocha and J. S. Lemos-Neto, \emph{New Cyclically Permutable Codes}, IEEE Information Theory Workshop, Paraty/Brazil, 693-697, 2011.


\bibitem{tuxanidy} A. Tuxanidy and Q. Wang, \emph{Compositional Inverses and Complete Mappings over Finite Fields}, Discrete Applied Mathematics, vol 217, 318-329,
2017.

\bibitem{compinverse2} B. Wu, \emph{The Compositional Inverse of a Class of Linearized Permutation Polynomials over $\mathbb{F}_{2^n}$, $n$ odd}, Finite Fields and Their Applications, vol 29, 34-48, 2014.

\bibitem{wuliu} B. Wu and Z. Liu, \emph{Linearized Polynomials over Finite Fields Revisited}, Finite Fields and Their Applications, vol 22, 79-100, 2013.

\bibitem{yuanzeng}
P. Yuan and X. Zeng, \emph{A Note on Linear Permutation Polynomials},
Finite Fields and Their Applications,
vol 17, 488-491, 2011.

\bibitem{zhao} W. Zhao and X. Tang, \emph{A Characterization of Cyclic Subspace Codes via Subspace Polynomials},
Finite Fields and Their Applications, vol 57, 1-12, 2019.

\bibitem{china} Y. Zheng, Q. Wang and W. Wei, \emph{On Inverses of Permutation Polynomials of Small Degree Over Finite Fields}, IEEE Transactions on Information Theory, vol 66, 914-922, 2020.

\bibitem{zheng} D. Zheng, M. Yuan, N. Li, L. Hu and X. Zeng, \emph{Constructions of Involutions Over Finite Fields}, IEEE Transactions on Information Theory, vol 65, 7876-7883, 2019.


\end{thebibliography}
%

\end{document}